\documentclass[preprint,12pt]{elsarticle}

\usepackage{graphicx}
\usepackage{amssymb}
\usepackage{amsthm}
\usepackage[fleqn]{amsmath}
\usepackage{mathtools}
\usepackage{cases}
\usepackage{lipsum}

\makeatletter
\def\ps@pprintTitle{%
 \let\@oddhead\@empty
 \let\@evenhead\@empty
 \def\@oddfoot{}%
 \let\@evenfoot\@oddfoot}
\makeatother

\biboptions{sort&compress}
\newtheorem{thm}{Theorem}

\theoremstyle{example}

\theoremstyle{definition}

\theoremstyle{remark}
\newtheorem{rem}{Remark}

\journal{}

\begin{document}

\begin{frontmatter}



\title{Precise interpretation of
the conformable fractional derivative}
\author{Ahmed A. Abdelhakim}
\address{Mathematics Department, Faculty of Science, Assiut University, Assiut 71516 - Egypt\\
Email: ahmed.abdelhakim@aun.edu.eg}
\begin{abstract}
Let $\alpha\in\,]0,1[$.
We prove that the existence of the conformable fractional derivative $T_{\alpha}f$
of a function $f:[0,\infty[\,\longrightarrow \mathbb{R}$ introduced by Khalil et al. in [R. Khalil, M. Al Horani, A. Yousef, M. Sababheh, A new definition of fractional derivative, J. Comput. Appl. Math. 264 (2014) 65-70] is equivalent to classical differentiability. Precisely
the fractional $\alpha$-derivative of $f$
is the pointwise product
$T_{\alpha}f(x)=x^{1-\alpha}f^{\prime}(x)$, $x>0$.
This simplifies the recent results
concerning conformable fractional calculus.
\end{abstract}
\end{frontmatter}
Let $0<\alpha<1$. The authors in \cite{khalil} proposed
defining the ``conformable fractional derivative"
$T_{\alpha}f$ of a function $f:[0,\infty[\,\longrightarrow \mathbb{R}$ by
\begin{equation}\label{q0}
T_{\alpha}f(x):=
\lim_{\epsilon\rightarrow 0}
\frac{f(x+\epsilon x^{1-\alpha})-f(x)}{\epsilon},\quad x>0,
\end{equation}
provided the limit exists, in which case $f$ is called $\alpha$-differentiable. And if $f$ is $\alpha$-differentiable on $\,]0,a[\,$ for some $a>0$, then $T_{\alpha}f(0):=\lim_{x\rightarrow 0^{+}}T_{\alpha}f(x)$. They also show
(\cite{khalil}, Theorem 2.2)
that if a function $f$ is differentiable at some $x>0$ then $T_{\alpha}f(x)$ exists and
\begin{equation}\label{q1}
T_{\alpha}f(x)=x^{1-\alpha}f^{\prime}(x).\\
\end{equation}
\indent We show that if $f$ is $\alpha$-differentiable
at some $x>0$ then it must be differentiable
at $x$ and (\ref{q1}) is satisfied. This means
that differentiability and $\alpha$-differentiability
in the sense of the existence of the limit
(\ref{q0}) are equivalent. In particular
a function $f:[0,\infty[\,\longrightarrow \mathbb{R}$ is $\alpha$-differentiable
at $x=0$ if and only if it is differentiable
on $\,]0,a[\,$ for some $a>0$, and
\begin{equation}\label{atz}
T_{\alpha}f(0)=\lim_{x\rightarrow 0^{+}}x^{1-\alpha}f^{\prime}(x).
\end{equation}
Thus
$T_{\alpha}f(0)=0$ for every function $f$
continuously differentiable on a neighbourhood of $0$
or, more gereally, differentiable on a right neighbourhood of 0 with a bounded derivative therein.
\\
\indent This equivalence between
conformable $\alpha$-differentiability and classical differentiability is already implicitly pointed out by
Tarasov \cite{Tarasov} where he proves
that the violation of the Leibniz rule,
$D^{\alpha}(fg)=f D^{\alpha}g+g D^{\alpha}f$,
is necessary for the order $\alpha$
of a fractional derivative $D^{\alpha}$
to be non-integer. The fractional
derivative $T_{\alpha}$ does satisfy the Leibniz rule
(\cite{khalil}, Theorem 2.2).
Our proof is far simpler and enables us to express
$T_{\alpha}f$ explicitly in terms of $f^{\prime}$
via (\ref{q1}).
Using this explicit pointwise relation would extremely
simplify many of the proofs/computational methods in
\cite{Abdeljawad,Benaoumeur,Chen,Eslami,Ghazala,
Hosseini,Katugampola,khalil,Morales,
Ortigueira,Tarasov,Yang,Zhou} to name a few.
\\
\indent It is claimed in \cite{khalil} that
an $\alpha$-differentiable function is not necessarily
differentiable.
We have not found any counterexamples
in the literature that support this claim except for the one example  (see \cite{khalil}) of $g(x):=\sqrt{x}$.
Of course, while $T_{\frac{1}{2}}g(0)=1$,
$g^{\prime}(0)$ does not exist. First, this does not apply to the translation of $g$ to any $x_{0}>0$. Simply consider
$h(x):=\sqrt{x-x_{0}}$ with $x_{0}>0$.
Then $T_{\alpha}h(x_{0})={x_{0}^{\frac{1-\alpha}{2}}}
\lim_{\epsilon\rightarrow 0}\frac1{
\sqrt{\epsilon}}$
does not exist for any $0\leq \alpha<1$, and neither does $h^{\prime}(x_{0})$. Second,
since the existence of $f^{\prime}(0)$
is independent of the existence of
$\lim_{x\rightarrow 0^{+}} x^{1-\alpha}
f^{\prime}(x)$, we realize from (\ref{atz})
that differentiability at 0 is independent of
$\alpha$-differentiability there. For instance,
the function $\tilde{g}:=x^2 \chi_{\mathbb{Q}}$
is differentiable only at 0, and therefore,
by (\ref{atz}),
$T_{\alpha}\tilde{g}(0)$ does not exist.
\begin{thm}
Fix $\,0<\alpha<1$, $x>0$. A function
$\,f:[0,\infty[\,\longrightarrow \mathbb{R}\,$
has a conformable fractional derivative of
order $\alpha$ at $x$ if and only if
it is differentiable at $x$ and
(\ref{q1}) holds.
\end{thm}
\begin{proof}
By Theorem 2.2 in \cite{khalil}, it suffices
to prove that if $T_{\alpha}f(x)$
exists then so does $f^{\prime}(x)$.
We have
\begin{eqnarray*}
\lim_{\epsilon\rightarrow 0}
\frac{f(x+\epsilon)-f(x)}{\epsilon}
 &=&\lim_{\epsilon\rightarrow 0} x^{\alpha-1}
\frac{f(x+(\epsilon\,x^{\alpha-1})\,x^{1-\alpha})-f(x)}{
\epsilon\,x^{\alpha-1} }   \\
   &=&x^{\alpha-1}\lim_{\epsilon\rightarrow 0}\,
\frac{f(x+\epsilon\,x^{1-\alpha})-f(x)}{
\epsilon}.
\end{eqnarray*}
\end{proof}
Another analogous attempt to define a conformable
$\alpha$-derivative of a function
$f:[0,\infty[\,\rightarrow \mathbb{R}$
appears in
\cite{Katugampola1} where
the $\alpha$-derivative takes the form:
\begin{equation*}
U_{\alpha}f(x):=
\lim_{\epsilon\rightarrow 0}
\frac{f(x e^{\epsilon x^{-\alpha}})-f(x)}{\epsilon},\quad x>0,
\end{equation*}
provided the limit exists,
$U_{\alpha}f(0):=\lim_{x\rightarrow 0^{+}}
x^{1-\alpha}U_{\alpha}f(x)$.
It is shown in (\cite{Katugampola1}, Theorem 2.3)
that if $f^{\prime}(x)$ exists
at some $x>0$ then so does
$U_{\alpha}f(x)$ and
\begin{equation}\label{qq1}
U_{\alpha}f(x)=x^{1-\alpha}f^{\prime}(x).
\end{equation}
It is claimed in \cite{Katugampola1}
that there exists a
conformable $\alpha$-differentiable
function that is not differentiable without providing
any counterexamples. With the exception of
the origin, this claim is also
false.
\begin{thm}\label{thm2}
Fix $\,0<\alpha<1$, $x>0$. A function
$\,f:[0,\infty[\,\longrightarrow \mathbb{R}\,$
has a conformable fractional derivative $U_{\alpha}f(x)$ if and only if
it is differentiable at $x$ and
(\ref{qq1}) holds.
\end{thm}
\begin{proof}
By Theorem 2.3 in \cite{Katugampola1}, we only need
to show that $f^{\prime}(x)$
exists whenever $U_{\alpha}f(x)$ does.
Since
\begin{equation*}
x+\epsilon=
e^{\log{\left(x+\epsilon\right)}}=
xe^{\log{\left(1+\frac{\epsilon}{x}\right)}}
= xe^{\epsilon
x^{-\alpha}
\left(x^{\alpha}\log{\left(1+\frac{\epsilon}{x}\right)^{
\frac{1}{\epsilon}}}\right)}
\end{equation*}
then
\begin{align*}
&\lim_{\epsilon\rightarrow 0}
\frac{f(x +\epsilon)-f(x)}{\epsilon}
 \;=\;\lim_{\epsilon\rightarrow 0}
\frac{f\left(xe^{\epsilon
x^{-\alpha}
\left(x^{\alpha}\log{\left(1+\frac{\epsilon}{x}\right)^{
\frac{1}{\epsilon}}}\right)}\right)-f(x)}{\epsilon}=\\
&=\;\lim_{\epsilon\rightarrow 0}
\frac{f\left(xe^{\epsilon
x^{-\alpha}
\left(x^{\alpha}\log{\left(1+\frac{\epsilon}{x}\right)^{
\frac{1}{\epsilon}}}\right)
}\right)-f(x)}{\epsilon\left(x^{\alpha}
\log{\left(1+\frac{\epsilon}{x}\right)^{
\frac{1}{\epsilon}}}\right)}\,\left(x^{
\alpha}\log{\left(1+\frac{\epsilon}{x}\right)^{
\frac{1}{\epsilon}}}\right).
\end{align*}
Observing that
\begin{equation*}
\lim_{\epsilon\rightarrow 0}
\log{\left(1+\frac{\epsilon}{x}\right)^{
\frac{1}{\epsilon}}}=\frac{1}{x}
\end{equation*}
we infer that
\begin{equation*}
\lim_{\epsilon\rightarrow 0}
\frac{f(x +\epsilon)-f(x)}{\epsilon}
\,=\,x^{\alpha-1}\lim_{\epsilon\rightarrow 0}
\frac{f\left(xe^{\epsilon
x^{-\alpha}}\right)-f(x)}{\epsilon}.
\end{equation*}
\end{proof}
\begin{rem}
Theorem \ref{thm2} substantially simplifies
the proofs in \cite{Anderson,Katugampola1}.
\end{rem}
We conclude with a general principle:
\begin{thm}
Suppose $h:\,]-1,1[\,\times\mathbb{R}\longrightarrow
\mathbb{R}$ is such that
$\,\lim_{\epsilon\rightarrow 0}h(\epsilon,x_{0})\neq 0$
for some $x_{0}\in \mathbb{R}$. Then a function $\phi:\mathbb{R}\longrightarrow
\mathbb{R}$ differentiable at $x_{0}$
if and only if the limit
\begin{equation*}
\tilde{\phi}(x_{0}):=\lim_{\epsilon\rightarrow 0}
\frac{\phi\left(x_{0}+\epsilon h(\epsilon,x_{0})\right)-\phi(x_{0})}{\epsilon}
\end{equation*}
exists, in which case $\tilde{\phi}(x_{0})=
\psi(x_{0})\phi^{\prime}(x_{0})$,
$\; \psi(x)=
\lim_{\epsilon\rightarrow 0}h(\epsilon,x)$.
\end{thm}

\textbf{References}
\bibliographystyle{model1a-num-names}
\bibliography{<your-bib-database>}

\end{document}